\documentclass[preprint,12pt]{elsarticle}
\usepackage[T1]{fontenc}
\usepackage[utf8]{inputenc}
\usepackage[english]{babel}
\usepackage{amsmath,amssymb,amsthm,mathtools}
\usepackage{geometry}
\usepackage{microtype}
\usepackage{enumitem}
\usepackage{hyperref}
\usepackage{cleveref}
\usepackage{comment}
\crefname{equation}{eq.}{eqs.}
\Crefname{equation}{Eq.}{Eqs.}
\hypersetup{colorlinks=true,linkcolor=blue,urlcolor=blue,citecolor=blue}
\newtheorem{theorem}{Theorem}
\newtheorem{lemma}{Lemma}
\newtheorem{corollary}{Corollary}

\theoremstyle{remark}

\newcommand{\e}{{e}}
\newcommand{\ba}{\begin{eqnarray}}
\newcommand{\ea}{\end{eqnarray}}
\newcommand{\bea}{\begin{eqnarray}}
\newcommand{\eea}{\end{eqnarray}}
\newcommand{\beq}{\begin{equation}}
\newcommand{\eeq}{\end{equation}}
\newcommand{\dtau}{\,d\tau}
\newcommand{\dz}{\,dz}
\newcommand{\lap}[1]{\mathcal{L}\left\{ #1 \right\}}
\newcommand{\pdv}[1]{{\frac{\partial}{\partial #1}}}

\makeatletter
\def\ps@pprintTitle{%
  \let\@oddhead\@empty
  \let\@evenhead\@empty
  \def\@oddfoot{\hfil\thepage\hfil}
  \let\@evenfoot\@oddfoot}
\makeatother

\begin{document}
\begin{frontmatter}
\title{Inequalities, identities, and bounds for \\  divided differences of the exponential function}
\begin{abstract}
\noindent 
Let $\exp[x_0,x_1,\dots,x_n]$ denote the divided difference of the exponential function.
\begin{enumerate}[label=(\roman*)]
\item 
We prove that exponential divided differences, as functions
of two nodes of arbitrary positive multiplicities with all
other nodes fixed, are strictly sign-regular of all orders,
with log-submodularity as the order-two case.
We also establish strict supermodularity.
\item We establish the four-point inequality
\[
\exp[a,a,b,c]\,\exp[d,d,b,c]+\exp[b,b,a,d]\,\exp[c,c,a,d]-\exp[a,b,c,d]^2 > 0,
\]
for all \(a,b,c,d\in\mathbb{R}\). 
Furthermore, we obtain a sharp quantitative generalization
to arbitrarily many pairs of nodes, with additional common nodes allowed.
\item We obtain sharp two-sided bounds for $\exp[x_0,\dots,x_n]$ at fixed mean and variance;
as a consequence, we derive an asymptotic expansion for
$n!\exp[x_0,\ldots,x_n]$ as $n\to\infty$ for node
sequences of bounded variance.
\item We present closed-form identities for divided differences of the exponential function, including a convolution identity and summation formulas for repeated arguments. 
\end{enumerate}
\end{abstract}
\author[aff1]{Qiulin Zeng} 
\author[aff2,aff3]{Nic Ezzell}
\author[aff2,aff3]{Arman Babakhani}
\author[aff2,aff3]{Itay Hen}
\author[aff3]{Lev Barash}
\affiliation[aff1]{organization={Department of Mathematics, University of Southern California},
            city={Los Angeles},
            postcode={90089}, 
            state={CA},
            country={USA}} 
\affiliation[aff2]{organization={Department of Physics and Astronomy, University of Southern California},
            city={Los Angeles},
            postcode={90089}, 
            state={CA},
            country={USA}}
\affiliation[aff3]{organization={Information Sciences Institute, University of Southern California},
            city={Marina del Rey},
            postcode={90292}, 
            state={CA},
            country={USA}}

\end{frontmatter}

\section{Introduction}

Divided differences of the exponential function lie at the intersection of analysis, operator theory, probability, and computation.
Analytically, the Hermite--Genocchi simplex formula
represents $\exp[x_0,\ldots,x_n]$ as an integral of
the exponential of a convex combination of the nodes
over the standard $n$-simplex, yielding positivity
and a~priori bounds for the divided differences,
with applications to interpolation remainders and
related extremal problems~\cite{mccurdy1984,deboor2005}.
In functional calculus for matrices and operators, divided differences serve as
the kernels of double operator integrals and, for higher-order variations, of multiple
operator integrals, providing a concise framework for perturbation theory and
operator inequalities~\cite{DaleckiiKrein,BirmanSolomyak1967,BhatiaMatrixAnalysis}.
The simplex representation also provides a
probabilistic interpretation of exponential divided
differences through the moment-generating functions
of Dirichlet averages~\cite{NeumanVanFleet1994}.

For matrix functions $f(A)$, Schur-based methods reduce $A$
to an upper (block) triangular form $T$ and evaluate $f(T)$
using recurrences such as the Schur--Parlett scheme.
Divided differences of various orders arise in explicit
formulas for the entries of $f(T)$, including confluent
divided differences when eigenvalues coincide
\cite{Parlett1976,Higham2008,barash2022}.
For $f(x)=e^x$, related block-triangular formulas
include those of Van Loan type~\cite{VanLoan1978}.
At a diagonalizable matrix, the Fr\'echet derivative
acts entrywise in an eigenbasis through the two-point
divided differences $f[\lambda_i,\lambda_j]$~\cite{Higham2008}.
Higher-order divided differences also occur in higher-order
Fr\'echet derivatives~\cite{Peller2006MOI,Peller2016MOISurvey}.
For example, second Fr\'echet derivatives are used to bound
level-2 condition numbers of matrix functions~\cite{HighamRelton2014}.

In time integration for stiff ODEs, exponential integrators
such as exponential time differencing and exponential
Runge--Kutta methods are built from the $\varphi$-functions,
which satisfy $\varphi_k(z)=\exp[0,\ldots,0,z]$
with $k$ zero arguments~\cite{CoxMatthews2002,KassamTrefethen2005,
HochbruckOstermann2010,NiesenWright2012}.
For approximating the action of the matrix exponential
or the $\varphi$-functions on a vector, Newton/Leja
polynomial interpolation involves higher-order divided
differences in its coefficients and interpolation
remainders, while rational Krylov methods provide
related approximation techniques
\cite{deboor2005,CaliariKandolfOstermannRainer2016LejaBEA,
KandolfOstermannRainer2014LejaResidual,
Guettel2013RationalKrylovReview}.

Accurate computation of exponential divided differences
can be challenging, particularly for closely spaced nodes,
where standard recurrence formulas may suffer from
cancellation and accumulation of rounding errors.
The algorithm introduced in~\cite{barash2020} allows
exponential divided differences to be updated at controlled computational
cost when nodes are added or removed, making it suitable
for adaptive or evolving-node scenarios.
Other approaches~\cite{mccurdy1984,Caliari2007DividedDifferences,Zivcovich2019}
aim to reduce error amplification through suitable computational
schemes or to enable stable use of these divided differences
in exponential integrators with less restrictive time-step choices.

Over the past decade, exponential divided differences
have played a central role in permutation matrix
representation quantum Monte Carlo (PMR-QMC),
where they enter the configuration weights~\cite{albash2017,hen2018,gupta2020,hen2019,hen2021,
barash2024,akaturk2024,babakhani2025,ezzell2025fidsus,ezzell2025advanced}.
This framework employs a universal set of QMC updates
applicable to a broad class of many-body Hamiltonians,
without requiring model-specific update design~\cite{barash2024,babakhani2025,ezzell2025fidsus}.
Related applications of exponential divided differences
include quantum algorithms for Hamiltonian simulation~\cite{chen2021,kalev2021quantum}
and representations of the Dyson series and
discrete-variable path integrals~\cite{kalev2021dyson,kalev2024}.

Exponential divided differences also appear in the theory
of means. For positive numbers $t_0,\dots,t_n$, the quantity
$n!\exp[\ln t_0,\dots,\ln t_n]$ defines a multivariate
extension of the logarithmic mean~\cite{Neuman1994,Merikoski2004},
which is a special case of Carlson's Dirichlet
averages~\cite{Carlson1977}. For $n=1$, it reduces to
the classical logarithmic mean
$(t_1-t_0)/(\ln t_1-\ln t_0)$.
Consequently, the sharp bounds derived in
Section~\ref{sec:bounds} also provide bounds for
multivariate logarithmic means in terms of the
geometric mean and the variance of the logarithms.

The present work identifies several properties of exponential divided differences that arise naturally in applications and, 
to our knowledge, have not previously been documented or proved in the literature. 

The paper is organized as follows. 
Section~\ref{sec:logsubmod} establishes sign-regularity of all orders, its log-submodularity consequence, and supermodularity of exponential divided differences.
In Section~\ref{sec:fourpoint} 
we establish a four-point inequality using auxiliary inequalities.
Section~\ref{sec:paired-nodes} gives an independent
sharp quantitative generalization for disjoint pairs of node indices.
In Section~\ref{sec:bounds} we derive sharp two-sided bounds at fixed mean
and variance and their asymptotics as the number of nodes grows. 
Section~\ref{sec:identities} is dedicated to several closed-form identities, including a convolution identity and repeated-argument summations.
We provide some concluding remarks in Section~\ref{sec:conc}.

\section{Sign-regularity and supermodularity of exponential divided differences}
\label{sec:logsubmod}

\noindent
Fix real nodes $a_1,\ldots,a_n$, where $n\geq0$, and integers $p,q\geq1$. We write
\[
K_n(x,y):=\exp[a_1,\ldots,a_n,x^{(p)},y^{(q)}],
\]
where the superscripts indicate multiplicities. These divided differences are
strictly positive for real arguments. 
We show that they satisfy a hierarchy of determinant inequalities of every order, 
including the log-submodularity inequality as the $2\times2$ case.

A kernel $A(s,t)$ is called \emph{totally nonnegative} if, 
for every integer $m\geq1$ and every choice of strictly increasing sequences 
$s_1<\cdots<s_m$ and $t_1<\cdots<t_m$, the $m\times m$ matrix 
$[A(s_i,t_j)]_{i,j=1}^m$ has a nonnegative determinant.
If all such determinants are
strictly positive for strictly increasing arguments, the kernel is
\emph{strictly totally positive}. A kernel whose minors instead have a
prescribed sign depending only on their order is called \emph{sign-regular}.
We use three standard facts from the theory of totally nonnegative kernels.

\begin{lemma}
\label{lem:TN-facts}
\begin{enumerate}[label=(\alph*)]
\item For every $a\in\mathbb{R}$, the one-sided exponential kernel
$H_a(s,t):=e^{a(s-t)}\mathbf{1}_{\{s\geq t\}}$ is totally nonnegative.
\item If $C(s,t)=\int A(s,z)B(z,t)\,dz$, then
\begin{equation}\label{eq:continuous-CB}
\det[C(s_i,t_j)]_{i,j=1}^m
=\int\limits_{z_1<\cdots<z_m}
\det[A(s_i,z_j)]_{i,j=1}^m\,
\det[B(z_i,t_j)]_{i,j=1}^m\,d\mathbf z,
\end{equation}
provided the integrals converge absolutely. In particular, the composition
of totally nonnegative kernels is totally nonnegative.
\item The kernel $e^{st}$ is strictly totally positive.
\end{enumerate}
\end{lemma}

\begin{proof}
(a) For increasing arguments, every minor of the step-function matrix
$[\mathbf{1}_{\{s_i\geq t_j\}}]$ is equal to zero or one, and the factors
$e^{as_i}$ and $e^{-at_j}$ merely rescale rows and columns by positive numbers.
(b) Identity~\eqref{eq:continuous-CB} is the basic
composition formula~\cite[Ch.~3, \S~1]{KarlinTP}, the continuous analogue
of the Cauchy--Binet formula.
It follows by expanding the determinants, applying Fubini's theorem,
and symmetrizing over the integration variables.
(c) The substitution $z_j=e^{t_j}$ turns $\det[e^{s_it_j}]$ into the
generalized Vandermonde determinant $\det[z_j^{s_i}]$, which is positive for
$s_1<\cdots<s_m$ and $0<z_1<\cdots<z_m$; see~\cite[Sec.~4.2]{PinkusTP}.
\end{proof}

\begin{theorem}[Strict sign-regularity of all orders]
\label{thm:KnSSR}
For any integer $m\geq1$ and strictly increasing real sequences
$x_1<\cdots<x_m$ and $y_1<\cdots<y_m$,
\begin{equation}\label{eq:SSR}
(-1)^{m(m-1)/2}
\det\big[K_n(x_i,y_j)\big]_{i,j=1}^m>0.
\end{equation}
Equivalently, the kernel $K_n(x,-y)$ is strictly
totally positive of all orders.
\end{theorem}

For example, for strictly increasing $x_i$ and $y_j$, the determinants of
orders $2$, $3$, and $4$ of $K_n(x,y)$ have signs $-$, $-$, and $+$,
respectively.

\begin{proof}
\emph{Step 1: an integral representation.}
Let $n\geq1$ and set \( F_n(x,y):=\exp[a_1,\ldots,a_n,x,y] \).
By the Hermite--Genocchi formula,
\[
F_n(x,y)=
\int\limits_{\substack{u,v\geq0\\u+v\leq1}}
e^{xu+yv}\,g_n(1-u-v)\,du\,dv,
\]
where, for $t\geq0$,
\begin{equation}\label{eq:gn}
g_n(t):=
\int\limits_{\substack{s_1+\cdots+s_n=t\\s_i\geq0}}
\exp\left(\sum_{i=1}^n a_i s_i\right)
\,ds_1\cdots ds_{n-1}.
\end{equation}
Here $s_n=t-\sum_{i=1}^{n-1}s_i$, and integration
is with respect to the first $n-1$ coordinates.
For $n=1$, the definition gives $g_1(t)=e^{a_1t}$.
We also set $g_n(t)=0$ for $t<0$.

The standard differentiation identity for divided
differences gives
\begin{equation}\label{eq:Kn-uv}
K_n(x,y)=
\frac{\partial_x^{p-1}\partial_y^{q-1}F_n(x,y)}
{(p-1)!\,(q-1)!}=
\frac{1}{(p-1)!\,(q-1)!}
\int_0^1\!\!\int_0^1
e^{xu+yv}\,W(u,v)\,du\,dv,
\end{equation}
where
\[
W(u,v):=u^{p-1}v^{q-1}g_n(1-u-v).
\]

\emph{Step 2: the sign of the determinants of $W$.}
By \Cref{eq:gn}, the kernel $g_n(s-t)$ is the
composition $H_{a_1}\circ\cdots\circ H_{a_n}$
of the one-sided exponential kernels 
$H_a(s,t):=e^{a(s-t)}\mathbf{1}_{\{s\geq t\}}$.
By Lemma~\ref{lem:TN-facts}(a) and~(b),
\begin{equation}\label{eq:gn-TP}
\det[g_n(s_i-t_j)]_{i,j=1}^m\geq0
\quad\text{for }s_1<\cdots<s_m,\ 
t_1<\cdots<t_m.
\end{equation}

For $0<u_1<\cdots<u_m$ and $0<v_1<\cdots<v_m$,
reversing the order of the rows (requiring $m(m-1)/2$ adjacent row swaps)
and extracting the positive row and column factors gives
\begin{equation}\label{eq:W-sign}
(-1)^{m(m-1)/2}\det[W(u_i,v_j)]_{i,j=1}^m
=
\left(\prod_{i=1}^m u_i^{p-1}v_i^{q-1}\right)
\det[g_n(1-u_{m+1-i}-v_j)]_{i,j=1}^m
\geq0,
\end{equation}
by \Cref{eq:gn-TP}, since $1-u_m<\cdots<1-u_1$.

\emph{Step 3: composition and strictness.}
Applying \Cref{eq:continuous-CB} twice to
\Cref{eq:Kn-uv}, we obtain
\begin{multline}
\label{eq:SSR-integral}
((p-1)!\,(q-1)!)^m \det[K_n(x_i,y_j)]_{i,j=1}^m
\\
=\iint\limits_{\substack{0<u_1<\cdots<u_m<1 \\ 0<v_1<\cdots<v_m<1}}
\det[e^{x_i u_j}]_{i,j=1}^m
\det[W(u_i,v_j)]_{i,j=1}^m
\det[e^{v_i y_j}]_{i,j=1}^m
\,d\mathbf v\,d\mathbf u.
\end{multline}
By Lemma~\ref{lem:TN-facts}(c), the two
exponential determinants are strictly positive.
By \Cref{eq:W-sign}, the integrand therefore
has the sign $(-1)^{m(m-1)/2}$ wherever
it is nonzero.

It remains to show that the integral in
\Cref{eq:SSR-integral} is nonzero. Choose
\(0<t_1<s_1<t_2<s_2<\cdots<t_m<s_m<1\)
and set $u_i=1-s_{m+1-i}$, $v_j=t_j$.
Since $g_n(t)>0$ for $t>0$ and vanishes for $t<0$,
the matrix $[g_n(s_i-t_j)]_{i,j=1}^m$ is lower
triangular, with all entries on and below the diagonal
strictly positive. Hence its determinant is positive.
The strict interlacing persists on an open neighborhood,
so by \Cref{eq:W-sign},
$\det[W(u_i,v_j)]$ has the required strict sign
on a set of positive measure.
Since \Cref{eq:W-sign} also excludes the opposite sign
throughout the integration domain, and the exponential
determinants in \Cref{eq:SSR-integral} are strictly positive,
\Cref{eq:SSR} follows for $n\geq1$.

\emph{The case $n=0$.}
Starting from
\(\exp[x,y]=\int_0^1 e^{x\tau+y(1-\tau)}\,d\tau\)
and differentiating, we obtain
\begin{equation}\label{eq:K0HG}
K_0(x,y)=
\frac{1}{(p-1)!\,(q-1)!}
\int_0^1 \tau^{p-1}(1-\tau)^{q-1}
e^{x\tau+y(1-\tau)}\,d\tau.
\end{equation}
By~\Cref{eq:continuous-CB}, we have
\begin{multline}
\label{eq:K0-CB}
((p-1)!\,(q-1)!)^m
\det[K_0(x_i,y_j)]_{i,j=1}^m
\\
=\int\limits_{0<\tau_1<\cdots<\tau_m<1}
\left(\prod_{i=1}^m
\tau_i^{p-1}(1-\tau_i)^{q-1}\right)
\det[e^{x_i\tau_j}]_{i,j=1}^m
\det[e^{(1-\tau_i)y_j}]_{i,j=1}^m
\,d\boldsymbol{\tau}.
\end{multline}
On the right-hand side, the first determinant is
strictly positive by Lemma~\ref{lem:TN-facts}(c),
while the second has the strict sign
$(-1)^{m(m-1)/2}$, since
$1-\tau_1>\cdots>1-\tau_m$.
This proves \Cref{eq:SSR} for $n=0$.
\end{proof}

\begin{corollary}[Log-submodularity]
\label{cor:KnTN2}
For any $n\geq0$, $p,q\geq1$ and $x_1\leq x_2$, $y_1\leq y_2$,
\begin{equation}\label{eq:TN2}
K_n(x_1,y_1)\,K_n(x_2,y_2)
\leq K_n(x_1,y_2)\,K_n(x_2,y_1),
\end{equation}
with strict inequality if $x_1<x_2$ and $y_1<y_2$. Equivalently, $\log K_n$ is
submodular. In particular, $K_n(x,x)\,K_n(y,y)\leq K_n(x,y)\,K_n(y,x)$ for all $x,y\in\mathbb{R}$.
Moreover, $\partial_x\partial_y\log K_n<0$; in terms of divided differences,
\begin{multline}\label{eq:TN2-infinitesimal}
\exp[a_1,\ldots,a_n,x^{(p+1)},y^{(q+1)}]\,\exp[a_1,\ldots,a_n,x^{(p)},y^{(q)}]\\
<\exp[a_1,\ldots,a_n,x^{(p+1)},y^{(q)}]\,\exp[a_1,\ldots,a_n,x^{(p)},y^{(q+1)}].
\end{multline}
\end{corollary}
\begin{proof}
For $x_1<x_2$ and $y_1<y_2$, \Cref{eq:TN2} is the case $m=2$ of \Cref{eq:SSR};
if $x_1=x_2$ or $y_1=y_2$, both sides coincide. 
Taking logarithms gives the submodularity of
$\log K_n$. Choosing $(x_1,y_1)=(x,x)$ and
$(x_2,y_2)=(y,y)$ proves the particular case for $x\leq y$;
invariance under interchanging $x$ and $y$ extends it to all $x,y$.

Differentiation with respect to a node of multiplicity $p$ multiplies the
divided difference by $p$ and raises the multiplicity by one, so that
$\partial_xK_n=p\exp[a_1,\ldots,a_n,x^{(p+1)},y^{(q)}]$,
$\partial_yK_n=q\exp[a_1,\ldots,a_n,x^{(p)},y^{(q+1)}]$ and
$\partial_x\partial_yK_n=pq\exp[a_1,\ldots,a_n,x^{(p+1)},y^{(q+1)}]$.
Hence \Cref{eq:TN2-infinitesimal} is equivalent to
$K_n\,\partial_x\partial_yK_n<\partial_xK_n\,\partial_yK_n$, i.e., to
$\partial_x\partial_y\log K_n<0$.

Although \Cref{eq:TN2} implies
$\partial_x\partial_y\log K_n\leq0$,
even its strict version does not necessarily imply
$\partial_x\partial_y\log K_n<0$ at every point.
To prove $\partial_x\partial_y\log K_n<0$ for $n\geq1$,
set $x_1=x$, $x_2=x+h$, $y_1=y$, $y_2=y+k$
in \Cref{eq:SSR-integral}, divide by $hk$, and let
$h,k\to0^+$.
The two exponential determinants, divided by $h$ and $k$,
converge to
$(u_2-u_1)e^{x(u_1+u_2)}$ and
$(v_2-v_1)e^{y(v_1+v_2)}$, respectively,
both strictly positive.
By \Cref{eq:W-sign}, the remaining determinant is
nonpositive everywhere and strictly negative on an open set,
as shown in the proof of Theorem~\ref{thm:KnSSR}.
Thus
$K_n\partial_x\partial_yK_n
-\partial_xK_n\,\partial_yK_n<0$.
For $n=0$, the same limiting argument applied to
\Cref{eq:K0-CB} gives an integrand proportional to
$-(\tau_2-\tau_1)^2$ times a strictly positive weight.
Hence the inequality is strict also for $n=0$.
\end{proof}

\begin{theorem}[Supermodularity]
\label{thm:KnSupermod}
For all $x_1\leq x_2$ and $y_1\leq y_2$,
\begin{equation}\label{eq:supermod}
K_n(x_1,y_1)+K_n(x_2,y_2)
\geq K_n(x_1,y_2)+K_n(x_2,y_1).
\end{equation}
The inequality is strict if $x_1<x_2$ and $y_1<y_2$.
\end{theorem}
\begin{proof}
Differentiating with respect to $x$ and $y$ gives
\[
\partial_x\partial_yK_n(x,y)
=pq\,\exp[a_1,\ldots,a_n,x^{(p+1)},y^{(q+1)}]>0.
\]
Integrating this mixed derivative over
$[x_1,x_2]\times[y_1,y_2]$ yields \Cref{eq:supermod} and its
strict version.
\end{proof}

\section{A Four-Point Inequality for Divided Differences}
\label{sec:fourpoint}

Let $\exp[x_0,x_1,\dots,x_n]$ denote the divided differences of the exponential function, and
\begin{align}
f(a,b,c,d) &= \exp[a,a,b,c]\exp[d,d,b,c] + \exp[b,b,a,d]\exp[c,c,a,d] - \exp[a,b,c,d]^2, \\
h(x,y) &= \exp[x,x,y,y](x-y)^2 = \e^x + \e^y - 2\exp[x,y].
\end{align}
Equivalently,
\beq
\label{eq:h-phi}
h(x,y)=2\e^{(x+y)/2}\,\phi\!\left(\frac{x-y}{2}\right),\qquad
\phi(u):=\cosh u-\frac{\sinh u}{u}\quad\text{with}\ \ \phi(0):=0.
\eeq

In this section, we prove the four-point inequality
\[
\exp[a,a,b,c]\exp[d,d,b,c] + \exp[b,b,a,d]\exp[c,c,a,d]
- \exp[a,b,c,d]^2 \geq 0,
\]
which holds for all $a,b,c,d \in\mathbb{R}$.
Strict positivity follows from Corollary~\ref{cor:paired-four-six}.

\begin{lemma}
\label{lemma:h}
For pairwise distinct real numbers $a,b,c,d$,
\beq
f(a,b,c,d) = \frac{h(a,c)h(b,d)+h(a,b)h(c,d)-h(b,c)h(a,d)}{2(a-b)(a-c)(b-d)(c-d)}.
\label{eq:statement1}
\eeq
\end{lemma}
\begin{proof}
We have
\begin{multline*}
f(a,b,c,d) = 
\frac{\e^a + \exp[b,c] - \exp[a,c] - \exp[a,b]}{(a-b)(a-c)}\times\frac{\e^d + \exp[b,c] - \exp[b,d] - \exp[c,d]}{(b-d)(c-d)}
\\ +
\frac{\e^b + \exp[a,d] - \exp[a,b] - \exp[b,d]}{(b-a)(b-d)}\times\frac{\e^c + \exp[a,d] - \exp[a,c] - \exp[c,d]}{(c-a)(c-d)}
\\ -
\frac{\exp[a,c]-\exp[a,d]-\exp[b,c]+\exp[b,d]}{(a-b)(c-d)}\times\frac{\exp[a,b]+\exp[c,d]-\exp[a,d]-\exp[b,c]}{(a-c)(b-d)}.
\end{multline*}
It follows that
\begin{multline}
f(a,b,c,d)(a-b)(a-c)(b-d)(c-d) = 
2\exp[a,c]\exp[b,d] - (\e^a+\e^c)\exp[b,d]
\\ - (\e^b+\e^d)\exp[a,c]
+ 2\exp[a,b]\exp[c,d] - (\e^a+\e^b)\exp[c,d]
\\ - (\e^c+\e^d)\exp[a,b] + \e^{b+c} + \e^{a+d}
\\ -
\left(
2\exp[a,d]\exp[b,c] - (\e^b+\e^c)\exp[a,d] - (\e^a+\e^d)\exp[b,c]
\right).
\label{eq:statement1a}
\end{multline}
Eq.~(\ref{eq:statement1}) follows from Eq.~(\ref{eq:statement1a}).
\end{proof}

\begin{lemma}
\label{lem:phi}
\beq
\label{ineq:phi}
\phi(x+y)\phi(y+z)\ \ge\ \phi(x)\phi(z) + \phi(y)\phi(x+y+z) \qquad \text{for all}\quad x,y,z\ge 0.
\eeq
\end{lemma}

\begin{proof}
Let $F(y)=\phi(x+y)\phi(y+z)-\phi(y)\phi(x+y+z)$.
We have 
\begin{equation}\label{eq:Fprime-G}
F'(y)=G_s(y+z)-G_s(y),
\end{equation}
where $s=x+2y+z$ and $G_s(t):=\phi'(s-t)\phi(t)+\phi(s-t)\phi'(t)$ for $t\in[0,s]$.

Since $\phi''(u) = (1+2/u^2)\phi(u)$ for $u\ne0$, we have, for $0<t<s$,
\[
G_s'(t) = \phi(s-t)\phi(t)\left(\frac{\phi''(t)}{\phi(t)}-\frac{\phi''(s-t)}{\phi(s-t)}\right) 
= 2\;\phi(s-t)\;\phi(t)\left(t^{-2}-(s-t)^{-2}\right).
\]
Therefore $G_s$ is strictly increasing on $[0,s/2]$, strictly decreasing on $[s/2,s]$,
and attains its maximum at $t=s/2$.
Moreover, $G_s$ is symmetric about $s/2$ since $G_s(s-t)=G_s(t)$.

Recall $s=x+2y+z$. Then
\[
\left|\,y-\frac{s}{2}\,\right|=\frac{x+z}{2},\qquad
\left|\,y+z-\frac{s}{2}\,\right|=\frac{|x-z|}{2}\;\le\;\frac{x+z}{2}.
\]
Thus $y+z$ lies no farther from the maximizer $s/2$ than $y$ does.
By the symmetry and unimodality of $G_s$, this implies
\[
G_s(y+z)\;\ge\;G_s(y).
\]
In view of Eq.~(\ref{eq:Fprime-G}), we have $F'(y)\ge 0$ for all $y\ge 0$.
Since $F$ is nondecreasing on $[0,\infty)$ and $F(0)=\phi(x)\phi(z)$, we obtain Eq.~(\ref{ineq:phi}).
\end{proof}

\begin{lemma}
\label{lemma:hpositive}
For any $a<b<c<d$ one has
\[
h(a,c)h(b,d) \ge h(b,c)h(a,d) + h(a,b)h(c,d).
\]
\end{lemma}

\begin{proof}
Let $x=(b-a)/2$, $y=(c-b)/2$, $z=(d-c)/2$. 
It follows from Eq.~(\ref{eq:h-phi}) that
\beq
h(a,c)h(b,d)-h(b,c)h(a,d)-h(a,b)h(c,d) = 4 e^{(a+b+c+d)/2} Q(x,y,z),
\eeq
where
\beq
Q(x,y,z) = \phi(x+y)\phi(y+z) - \phi(y)\phi(x+y+z) - \phi(x)\phi(z).
\eeq
It follows from Lemma~\ref{lem:phi} that 
\[
Q(x,y,z)\ \ge\ 0,
\]
and therefore the lemma follows.
\end{proof}

\begin{theorem}[Four-point inequality]
\label{thm:fourpoint}
For all $a,b,c,d \in \mathbb{R}$ one has
\beq
\label{eq:inequality}
\exp[a,a,b,c]\exp[d,d,b,c] + \exp[b,b,a,d]\exp[c,c,a,d] - \exp[a,b,c,d]^2 \geq 0.
\eeq
\end{theorem}
\begin{proof}
$f(a,b,c,d)$ is invariant under the eight-element subgroup $G$ of the symmetric group $S_4$ acting on $\{a,b,c,d\}$:
$G=\{\mathrm{id},\ (ad),\ (bc),\ (ad)(bc),\ (ab)(cd),\ (ac)(bd),\ (abdc),\ (acdb)\}$.
Hence the $24$ permutations of $(a,b,c,d)$ split into three $G$-orbits.
Let $\alpha<\beta<\gamma<\delta$ be the increasing rearrangement of $\{a,b,c,d\}$.
It suffices to verify that 
$f(\alpha,\beta,\gamma,\delta) \ge0$, $f(\alpha,\gamma,\delta,\beta) \ge0$, and $f(\alpha,\beta,\delta,\gamma) \ge0$.
For the three orderings, the denominators
in Lemma~\ref{lemma:h} have signs $+,+,-$,
respectively.
Lemma~\ref{lemma:hpositive}, together with $h\ge0$,
implies that the corresponding numerators are
nonnegative, nonnegative, and nonpositive.
Thus $f\ge0$ in each case.
If some of the points coincide, the result follows by continuity of divided differences.
\end{proof}

\section{A Generalization of the Four-Point Inequality}
\label{sec:paired-nodes}

We now establish a sharp paired-node inequality 
for exponential divided differences, allowing any number of disjoint pairs 
and additional unpaired nodes.

Write $D_n(x)=\exp[x_1,\ldots,x_n]$ for $x\in\mathbb R^n$, and let
$x^{(i\to j)}$ replace $x_j$ by $x_i$.  Since divided differences are
symmetric, we may label the disjoint pairs $(1,2),\ldots,(2m-1,2m)$ and
set
\begin{equation}\label{eq:paired-Pk}
P_k(x)=D_n\bigl(x^{(2k-1\to2k)}\bigr)
       D_n\bigl(x^{(2k\to2k-1)}\bigr),\qquad k=1,\ldots,m.
\end{equation}
All nodes other than $x_{2k-1}$ and $x_{2k}$, in particular the
unpaired nodes $x_{2m+1},\ldots,x_n$, occur in both factors.

\begin{theorem}[Sharp paired-node inequality]\label{thm:paired-nodes}
Let $m\ge1$, $n\ge2m$, and $x\in\mathbb R^n$.  Let $y$ be obtained
by replacing each pair $(x_{2k-1},x_{2k})$ by its arithmetic mean in
both positions, and put $\bar x=n^{-1}\sum_i x_i$. Then
\begin{equation}\label{eq:paired-main}
\sum_{k=1}^m P_k(x)-(m-1)D_n(x)^2
\ \ge\ D_n(y)^2
\ \ge\ \frac{e^{2\bar x}}{((n-1)!)^2}>0.
\end{equation}
Equality in the first inequality holds precisely when $x=y$, except
for $(n,m)=(2,1)$, when it holds for every $x$. Equality in the second
holds precisely when all coordinates of $y$ coincide. 
\end{theorem}

We need one elementary inequality for convolutions of intervals.
For $n\ge3$ and $c=(c_1,\ldots,c_n)\in(0,\infty)^n$ let
\[
\chi_i=\mathbf1_{[-c_i,c_i]},\quad
\rho=\mathop{\ast}_{i=1}^n\chi_i,\quad
\rho_i=\mathop{\ast}_{r\ne i}\chi_r,\quad
\rho_{ij}=\mathop{\ast}_{r\notin\{i,j\}}\chi_r.
\]
All these convolutions are even and nonincreasing on $[0,\infty)$.
Indeed, if a bounded function $g$ is even and nonincreasing on
$[0,\infty)$, then $g*\chi_i$ is even, and for $0\le u<v$,
\[
(g*\chi_i)(v)-(g*\chi_i)(u)
=\int_u^v [g(t+c_i)-g(t-c_i)]\,dt\le0,
\]
since $|t+c_i|\ge|t-c_i|$. Induction on the number of factors
gives the claim. For $n\ge3$, each $\rho_i$ is a convolution
of at least two indicators, so $\rho_i(u)$ depends continuously
on $u$ and $c$.

\begin{lemma}\label{lem:paired-kernel}
For $n\ge3$ and $n\ge2m$, define $\delta_k=|c_{2k-1}-c_{2k}|$.
Then
\begin{equation}\label{eq:paired-kernel}
H(c):=\sum_{k=1}^m4c_{2k-1}c_{2k}\,
\rho_{2k-1,2k}(\delta_k)-(m-1)\rho(0)\ge0.
\end{equation}
\end{lemma}
\begin{proof}
Put $R=\rho(0)$ and $A_i=2c_i\rho_i(c_i)$.  Since $\rho=\chi_i*\rho_i$
and $\rho_i$ is even, $R=\int_{-c_i}^{c_i}\rho_i(u)\,du$; hence
\[
R-A_i=\int_{-c_i}^{c_i}\bigl(\rho_i(u)-\rho_i(c_i)\bigr)\,du\ge0,
\]
because $\rho_i$ is nonincreasing on $[0,\infty)$. Replacing $c$ by
$\kappa c$ with $\kappa>0$ and substituting $u\mapsto\kappa u$ in each of
the $n-1$ convolution integrals multiplies $R$ by $\kappa^{n-1}$, so $R$
is homogeneous of degree $n-1$ in $c$. Since $\rho_i$ does not depend on
$c_i$, the formula for $R$ gives
$\partial_{c_i}R=\rho_i(c_i)+\rho_i(-c_i)=2\rho_i(c_i)$, which is
continuous in $c$. Thus $R$ is continuously differentiable, and Euler's
identity for homogeneous functions gives
\begin{equation}\label{eq:paired-budget}
\sum_{i=1}^n A_i= \sum_{i=1}^n c_i \partial_{c_i} R = (n-1)R.
\end{equation}
Now fix $i\ne j$. 
By interchanging the indices if necessary,
we may assume $c_i\ge c_j$; write $\delta=c_i-c_j$.
The convolution $\chi_i*\chi_j$ is trapezoidal: it equals $2c_j$ for
$|u|\le\delta$, $c_i+c_j-|u|$ for $\delta<|u|\le c_i+c_j$, and zero
elsewhere. Since $\rho_i=\chi_j*\rho_{ij}$, $\rho_j=\chi_i*\rho_{ij}$,
$\rho=(\chi_i*\chi_j)*\rho_{ij}$, and $\rho_{ij}$ is even,
\[
\begin{aligned}
A_i+A_j
&=4c_j\int_0^\delta \rho_{ij}(u)\,du
  +2(c_i+c_j)\int_\delta^{c_i+c_j}\rho_{ij}(u)\,du,\\
R
&=4c_j\int_0^\delta \rho_{ij}(u)\,du
  +2\int_\delta^{c_i+c_j}(c_i+c_j-u)\rho_{ij}(u)\,du.
\end{aligned}
\]
Hence $A_i+A_j-R=2\int_\delta^{c_i+c_j}u\,\rho_{ij}(u)\,du$. Since
$4c_ic_j=(c_i+c_j)^2-\delta^2=2\int_\delta^{c_i+c_j}u\,du$, it follows
that
\begin{equation}\label{eq:paired-pair-positive}
4c_ic_j\rho_{ij}(\delta)+R-A_i-A_j
=2\int_\delta^{c_i+c_j}
  u\,[\rho_{ij}(\delta)-\rho_{ij}(u)]\,du\ge0,
\end{equation}
because $\rho_{ij}$ is nonincreasing on $[0,\infty)$.
Finally, \Cref{eq:paired-budget} yields the exact decomposition
\[
H(c)=\sum_{k=1}^m
 \bigl[4c_{2k-1}c_{2k}\rho_{2k-1,2k}(\delta_k)
       +R-A_{2k-1}-A_{2k}\bigr]
 +\sum_{i=2m+1}^n(R-A_i)\ge0.
\]
\end{proof}

\begin{proof}[Proof of \Cref{thm:paired-nodes}]
For $n=2$, we have $m=1$ and $P_1(x)=D_2(y)^2=e^{2\bar x}$,
so both inequalities are equalities.
Assume henceforth $n\ge3$.  Equip
$\Delta=\{\lambda\in\mathbb R_{\ge0}^n:\sum_i\lambda_i=1\}$ and
$2\Delta=\{\nu\in\mathbb R_{\ge0}^n:\sum_i\nu_i=2\}$ with the
coordinate measures $d\lambda=d\lambda_1\cdots d\lambda_{n-1}$ and
$d\nu=d\nu_1\cdots d\nu_{n-1}$. The Hermite--Genocchi formula gives
$D_n(x)=\int_\Delta e^{\langle\lambda,x\rangle}\,d\lambda$. For an
interior point $\nu\in2\Delta$ we use the notation introduced before
\Cref{lem:paired-kernel} with $c_i=\nu_i/2$, so that $\sum_ic_i=1$.

We first establish the following integral representations:
\begin{align}
D_n(x)^2
 &=\int_{2\Delta}e^{\langle\nu,x\rangle}\rho(0)\,d\nu,
 \label{eq:paired-intD}\\
P_k(x)
 &=\int_{2\Delta}e^{\langle\nu,x\rangle}
    4c_{2k-1}c_{2k}\rho_{2k-1,2k}(\delta_k)\,d\nu.
 \label{eq:paired-intP}
\end{align}

For the first identity, multiply two Hermite--Genocchi integrals
and put $\nu=\lambda+\mu$, where $\lambda,\mu\in\Delta$.
For fixed $\nu$, the admissible pairs $(\lambda,\mu)$ satisfy
$0\le\lambda_i\le\nu_i$ and $\sum_i\lambda_i=1$.
Writing $\lambda_i=c_i+u_i$ gives $-c_i\le u_i\le c_i$ and
$\sum_i u_i=1-\sum_i c_i=0$.
By the definition of convolution, the coordinate volume of 
the set defined by \(|u_i|\le c_i\) and \(\sum_i u_i=0\) equals \(\rho(0)\).

For the second identity, let $(i,j)=(2k-1,2k)$ and put
$\nu_i=\lambda_i+\lambda_j$,
$\nu_j=\mu_i+\mu_j$, and
$\nu_r=\lambda_r+\mu_r$ for $r\notin\{i,j\}$.
For fixed $\nu$, the variables $\lambda_i$ and $\mu_i$
range independently over $[0,\nu_i]$ and $[0,\nu_j]$,
respectively, contributing the factor
$\nu_i\nu_j=4c_ic_j$.
The remaining coordinates satisfy
$0\le\lambda_r\le\nu_r$ and
$\sum_{r\notin\{i,j\}}\lambda_r=1-\nu_i$.
Writing $\lambda_r=c_r+u_r$, we obtain
$\sum_{r\notin\{i,j\}}u_r
=1-\nu_i-\sum_{r\notin\{i,j\}}c_r
=c_j-c_i$,
since $\sum_r c_r=1$.
The remaining section has
coordinate volume $\rho_{ij}(c_j-c_i)$.
Since $\rho_{ij}$ is even, this equals $\rho_{ij}(\delta_k)$.
Thus the total coordinate volume is
$4c_ic_j\rho_{ij}(\delta_k)$.

On each simplex, the sum constraint allows any one coordinate
to be omitted from the integration variables, without changing
the coordinate measure. Since $n\ge3$, we may choose an omitted
coordinate outside $\{i,j\}$. With this choice, both changes
of variables have Jacobian of absolute value one, completing
the proof of \Cref{eq:paired-intD,eq:paired-intP}.

Write $G(x)=\sum_kP_k(x)-(m-1)D_n(x)^2$. By
\Cref{eq:paired-intD,eq:paired-intP,eq:paired-kernel},
\begin{equation}\label{eq:paired-positive-Laplace}
G(x)=\int_{2\Delta}e^{\langle\nu,x\rangle}H(\nu/2)\,d\nu,
\qquad H(\nu/2)\ge0.
\end{equation}
The density $H(\nu/2)$ is
invariant under each interchange $\nu_{2k-1}\leftrightarrow\nu_{2k}$.
Writing $x_{2k-1}=s_k+t_k$ and $x_{2k}=s_k-t_k$, we have
$\langle\nu,x\rangle=\langle\nu,y\rangle
+\sum_kt_k(\nu_{2k-1}-\nu_{2k})$; averaging the integral over the $2^m$
interchanges, we obtain the key representation
\begin{equation}\label{eq:paired-cosh}
G(x)=\int_{2\Delta}e^{\langle\nu,y\rangle}H(\nu/2)
\prod_{k=1}^m
\cosh\!\bigl(t_k(\nu_{2k-1}-\nu_{2k})\bigr)\,d\nu.
\end{equation}
Since each hyperbolic cosine is at least one and
$H(\nu/2)\ge0$, we obtain
$G(x)\ge G(y)=D_n(y)^2$; the last equality holds because
$P_k(y)=D_n(y)^2$ for every $k$.
If some $t_k\ne0$, the corresponding hyperbolic cosine is
strictly greater than one 
except on the measure-zero set $\nu_{2k-1}=\nu_{2k}$.
Since the nonnegative weight
$e^{\langle\nu,y\rangle}H(\nu/2)$ has positive integral
$G(y)=D_n(y)^2$, the inequality is then strict.

Applying Jensen's inequality to the Hermite--Genocchi formula,
we obtain $D_n(y)\ge e^{\bar x}/(n-1)!$,
with equality if and only if all coordinates of $y$ coincide.
This last inequality, including its equality condition,
also follows from~\cite[Theorem~2.1]{FarwigZwick1985}.
This completes the proof.
\end{proof}

\begin{corollary}[Four- and six-point bounds]\label{cor:paired-four-six}
For all $a,b,c,d\in\mathbb R$, the function $f$ of
\Cref{sec:fourpoint} satisfies
\begin{equation}\label{eq:paired-four}
f(a,b,c,d)
\ge\exp\left[\frac{a+d}{2},\frac{a+d}{2},\frac{b+c}{2},\frac{b+c}{2}\right]^2
\ge\frac{e^{(a+b+c+d)/2}}{36}>0.
\end{equation}
In particular, the inequality in Theorem~\ref{thm:fourpoint} is strict.

\noindent
For all $a,b,c,d,r,s\in\mathbb R$,
\begin{multline}\label{eq:paired-six}
\exp[a,a,b,c,r,s]\exp[d,d,b,c,r,s]
+\exp[b,b,a,d,r,s]\exp[c,c,a,d,r,s]\\
+\exp[a,b,c,d,r,r]\exp[a,b,c,d,s,s]
-2\exp[a,b,c,d,r,s]^2 \\
\ge\exp\left[\frac{a+d}{2},\frac{a+d}{2},\frac{b+c}{2},\frac{b+c}{2},\frac{r+s}{2},\frac{r+s}{2}\right]^2
\ge\frac{e^{(a+b+c+d+r+s)/3}}{14400}>0.
\end{multline}
In \Cref{eq:paired-four}, equality holds in the first inequality if and
only if $a=d$ and $b=c$, and in the second if and only if $a+d=b+c$.
Similarly, in \Cref{eq:paired-six}, 
equality holds in the first inequality
if and only if $a=d$, $b=c$, and $r=s$, and in the second
if and only if $a+d=b+c=r+s$.
\end{corollary}
\begin{proof}
Use \Cref{thm:paired-nodes} with $x=(a,d,b,c)$, $m=2$, and
$x=(a,d,b,c,r,s)$, $m=3$.
\end{proof}

\section{A Sharp Two-sided Bound for Exponential Divided Differences at Fixed Mean and Variance}
\label{sec:bounds}

We derive sharp bounds for exponential divided differences when the mean and variance of their nodes are fixed.
Throughout this section $n\ge1$, and we write
\begin{equation}
\label{eq:notations}
\mu:=\frac1{n+1}\sum_{i=0}^n x_i,\qquad
\sigma^2:=\frac1{n+1}\sum_{i=0}^n(x_i-\mu)^2,\qquad
a_0:=\sqrt n\,\sigma,\qquad a:=\frac{n+1}{\sqrt n}\,\sigma.
\end{equation}
Define the positive functions
\begin{equation}\label{eq:LM-integrals}
L_n(\sigma):=n e^{-a_0}\int_0^1s^{n-1}e^{as}\,ds,\qquad
M_n(\sigma):=n e^{a_0}\int_0^1s^{n-1}e^{-as}\,ds.
\end{equation}
Equivalently, for $\sigma>0$,
$L_n(\sigma)=n e^{-a_0}(-a)^{-n}\gamma(n,-a)$,
$M_n(\sigma)=n e^{a_0}a^{-n}\gamma(n,a)$,
where $\gamma(n,z)=\int_0^z t^{n-1}e^{-t}\,dt$ is the lower incomplete gamma function.
The integral formulas give $L_n(0)=M_n(0)=1$ without taking limits.

\begin{theorem}[Sharp two-sided bound]\label{thm:sandwich}
For any $x_0,\dots,x_n\in\mathbb R$ with mean $\mu$ and variance $\sigma^2$,
\begin{equation}\label{eq:sandwich}
e^\mu L_n(\sigma)\;\le\;n!\,\exp[x_0,\ldots,x_n]
\;\le\;e^\mu M_n(\sigma).
\end{equation}
Both bounds are attained. 
Equality on the right holds precisely 
at permutations of $(\mu+a_0,\ \mu-a_0/n,\dots,\mu-a_0/n)$,
and equality on the left precisely at permutations of $(\mu-a_0,\ \mu+a_0/n,\dots,\mu+a_0/n)$.
\end{theorem}

\begin{proof}
Set $y_i=x_i-\mu$ and $\Phi(y):=n!\exp[y_0,\ldots,y_n]$, so that
$n!\exp[x_0,\ldots,x_n]=e^\mu\Phi(y)$. By the Hermite--Genocchi formula,
\begin{equation}\label{eq:HG}
\Phi(y)=n!\int_{\Delta_n}\exp\!\left(\sum_{i=0}^n\lambda_i y_i\right)d\lambda,
\end{equation}
where $\Delta_n=\{\lambda\in\mathbb R^{n+1}:\lambda_i\ge0,\ \sum_i\lambda_i=1\}$ and
$d\lambda=d\lambda_1\cdots d\lambda_n$, so that $n!\,d\lambda$ is the uniform
probability measure on $\Delta_n$. Thus $\Phi$ is a smooth symmetric function.

We first evaluate $\Phi$ at the candidate extremizers. Under the uniform probability
measure on $\Delta_n$, the coordinate $\lambda_0$ has density $n(1-s)^{n-1}$ on
$[0,1]$, so $1-\lambda_0$ has density $ns^{n-1}$. For a vector with one node $b$ and
$n$ nodes equal to $t$, the exponent in Eq.~(\ref{eq:HG}) equals
$b-(b-t)(1-\lambda_0)$; hence
\begin{equation}\label{eq:two-level}
\Phi(b,t,\ldots,t)=n e^{b}\int_0^1 s^{n-1}e^{-(b-t)s}\,ds.
\end{equation}
Let
\[
y^+=\left(a_0,-\frac{a_0}{n},\ldots,-\frac{a_0}{n}\right),\qquad
y^-=\left(-a_0,\frac{a_0}{n},\ldots,\frac{a_0}{n}\right).
\]
Since $a_0+a_0/n=a$, Eqs.~(\ref{eq:LM-integrals}) and~(\ref{eq:two-level}) give
$\Phi(y^+)=M_n(\sigma)$ and $\Phi(y^-)=L_n(\sigma)$.

The vector $y$ lies in the set
$\mathcal S=\left\{y\in\mathbb R^{n+1}:\sum_i y_i=0,\quad
\sum_i y_i^2=(n+1)\sigma^2\right\}$,
which is compact and contains $y^+$ and $y^-$. If $\sigma=0$, then
$\mathcal S=\{0\}$ and $y^\pm=0$. If $n=1$, then $\mathcal S=\{y^+,y^-\}$, and $y^-$
is a permutation of $y^+$. In both cases the theorem follows. We henceforth assume
$\sigma>0$ and $n\ge2$. Since $\Phi$ attains its maximum and minimum on $\mathcal S$,
it suffices to show that every local maximizer of $\Phi$ on $\mathcal S$ is a
permutation of $y^+$ and every local minimizer is a permutation of $y^-$.

Let $y$ be a local extremum of $\Phi$ on $\mathcal S$.
Since $y\ne0$ and $\sum_i y_i=0$, the gradients
$\mathbf 1=(1,\ldots,1)$ and $2y$ of the two constraints
are linearly independent. Hence the Lagrange multiplier
rule gives $\partial\Phi/\partial y_i=\alpha+\beta y_i$
for some $\alpha,\beta\in\mathbb R$.

Introduce $G(t):=n!\exp[y_0,\ldots,y_n,t]$.
By the differentiation identities for divided differences,
$\partial\Phi/\partial y_i=G(y_i)$. Hence,
\begin{equation}\label{eq:stationary-G}
G(y_i)=\alpha+\beta y_i\qquad (0\le i\le n).
\end{equation}
Since $G''(t)=2n!\exp[y_0,\ldots,y_n,t,t,t]>0$,
$G$ is strictly convex, and the equation $G(t)=\alpha+\beta t$ has at most two solutions. 
As $\sigma>0$ and $\sum_i y_i=0$, the coordinates of $y$ cannot all be equal. 
Hence they take exactly two distinct values $u>v$.
Since $G'$ is strictly increasing,
the mean value theorem and \Cref{eq:stationary-G} give
\begin{equation}\label{eq:chord-G}
G'(v)<\beta<G'(u).
\end{equation}

It remains to determine how many coordinates of $y$ equal $u$
and how many equal $v$. Suppose $y_i=y_j=t$ for distinct
indices $i,j$, and put $w=\mathbf e_i-\mathbf e_j$.
Since $\sum_k w_k=0$, $\sum_k y_kw_k=y_i-y_j=0$, and $|w|^2=2$,
the curve
\[
\eta(\theta)=y \cos\theta +
w \frac{|y|}{\sqrt{2}}\sin\theta
\]
lies in $\mathcal S$, with $\eta(0)=y$.

Since $\eta'(0)=|y|w/\sqrt2$ and $\eta''(0)=-y$,
the chain rule gives
\begin{equation}
\label{eq:ddPhi}
\frac{d^2}{d\theta^2}\Phi(\eta(\theta))
\bigg|_{\theta=0}
=\frac{|y|^2}{2}w^{\mathsf T}\nabla^2\Phi(y)w
  -\langle\nabla\Phi(y),y\rangle
=|y|^2\bigl(G'(t)-\beta\bigr).
\end{equation}
For the second equality, the differentiation identities
for divided differences give
$\partial^2\Phi/\partial y_i^2 = \partial^2\Phi/\partial y_j^2 = 2G'(t)$,
$\partial^2\Phi/\partial y_i\partial y_j = G'(t)$, and
$w^{\mathsf T}\nabla^2\Phi(y)w=2G'(t)$,
while \Cref{eq:stationary-G} and $\sum_k y_k=0$ give
$\langle\nabla\Phi(y),y\rangle=\beta|y|^2$.

By \Cref{eq:chord-G,eq:ddPhi}, the second derivative is
positive for $t=u$ and negative for $t=v$.
Hence, if two coordinates equal $u$, the point cannot be a local maximum; 
if two coordinates equal $v$, it cannot be a local minimum.

Consequently, at a maximum exactly one coordinate equals $u$.
The constraints $u+nv=0$ and $u^2+nv^2=(n+1)\sigma^2$
then give $u=a_0$ and $v=-a_0/n$, so $y$ is a permutation
of $y^+$. Similarly, at a minimum exactly one coordinate
equals $v$, giving $v=-a_0$ and $u=a_0/n$, so $y$ is a
permutation of $y^-$.
Thus, the maximum and minimum of $\Phi$ on $\mathcal S$
are $M_n(\sigma)$ and $L_n(\sigma)$, respectively.
Multiplying by $e^\mu$ proves \Cref{eq:sandwich}
and its equality cases.
\end{proof}

The lower bound in \Cref{thm:sandwich} sharpens the classical bound
$n!\exp[x_0,\ldots,x_n]\ge e^{\mu}$ of Farwig and
Zwick~\cite[Theorem~2.1]{FarwigZwick1985} whenever $\sigma>0$,
since then $L_n(\sigma)>1$.

\begin{lemma}\label{lem:LnMn-asymp}
For every $S>0$, uniformly for $0\le\sigma\le S$,
\begin{align}
\label{eq:Ln}
L_n(\sigma)&=1+\frac{\sigma^2}{2n}-\frac{\sigma^3}{3n^{3/2}}+O(n^{-2}),\\
\label{eq:Mn}
M_n(\sigma)&=1+\frac{\sigma^2}{2n}+\frac{\sigma^3}{3n^{3/2}}+O(n^{-2})
\end{align}
as $n\to\infty$.
\end{lemma}
\begin{proof}
Set $\varepsilon=\sigma/\sqrt n$, so that
$a_0=n\varepsilon$ and $a=(n+1)\varepsilon$.
Replacing $s$ by $1-s$ in Eq.~(\ref{eq:LM-integrals})
gives
$
M_n(\sigma)
=n e^{-\varepsilon}
  \int_0^1(1-s)^{n-1}e^{as}\,ds$,
$L_n(\sigma)
=n e^{\varepsilon}
  \int_0^1(1-s)^{n-1}e^{-as}\,ds$.
Expanding the exponentials in powers of $s$ and
integrating termwise with
$n\int_0^1(1-s)^{n-1}s^j\,ds=j!\,n!/(n+j)!$, we obtain
\[
M_n(\sigma)=e^{-\varepsilon}\sum_{j=0}^\infty c_{n,j}\,\varepsilon^j,\qquad
L_n(\sigma)=e^{\varepsilon}\sum_{j=0}^\infty c_{n,j}\,(-\varepsilon)^j,\qquad
c_{n,j}:=\frac{n!(n+1)^j}{(n+j)!}=\prod_{i=1}^j\frac{n+1}{n+i}.
\]
By applying the inequality $1+x\le e^x$ twice, we obtain
\[
1\ge c_{n,j}=\prod_{i=1}^j\Bigl(1+\frac{i-1}{n+1}\Bigr)^{-1}
\ge\exp\Bigl(-\frac{j(j-1)}{2(n+1)}\Bigr)\ge1-\frac{j(j-1)}{2(n+1)}.
\]
Comparing with the geometric series $\sum_j(\pm\varepsilon)^j=(1\mp\varepsilon)^{-1}$,
we find, for $|\varepsilon|<1$,
\[
\left|\sum_{j=0}^\infty c_{n,j}(\pm\varepsilon)^j-\frac{1}{1\mp\varepsilon}\right|
\le\sum_{j=2}^\infty\frac{j(j-1)}{2(n+1)}\,|\varepsilon|^j
=\frac{\varepsilon^2}{(n+1)(1-|\varepsilon|)^3}.
\]
For $0\le\sigma\le S$ we have $0\le\varepsilon\le S/\sqrt n$, so the right-hand side is
$O(n^{-2})$ uniformly, and
\[
M_n(\sigma)=\frac{e^{-\varepsilon}}{1-\varepsilon}+O(n^{-2}),\qquad
L_n(\sigma)=\frac{e^{\varepsilon}}{1+\varepsilon}+O(n^{-2}).
\]
Finally, since $-\log(1\mp\varepsilon)=\sum_{k\ge1}(\pm\varepsilon)^k/k$,
\[
\frac{e^{\mp\varepsilon}}{1\mp\varepsilon}
=\exp\Biggl(\sum_{k=2}^\infty\frac{(\pm\varepsilon)^k}{k}\Biggr)
=1+\frac{\varepsilon^2}{2}\pm\frac{\varepsilon^3}{3}+O(\varepsilon^4).
\]
With $\varepsilon^2=\sigma^2/n$, $\varepsilon^3=\sigma^3/n^{3/2}$ and
$\varepsilon^4=O(n^{-2})$, this proves Eqs.~(\ref{eq:Ln}) and~(\ref{eq:Mn}).
\end{proof}

\begin{corollary}[Large-node-count asymptotics]\label{thm:ddinfty}
Let $(x_i)_{i=0}^\infty$ be a sequence of real numbers, and write
\[
\mu_n:=\frac1{n+1}\sum_{i=0}^n x_i,\qquad
\sigma_n^2:=\frac1{n+1}\sum_{i=0}^n(x_i-\mu_n)^2.
\]
If $\sup_{n\ge0}\sigma_n^2<\infty$, then
\[
n!\,\exp[x_0,\ldots,x_n]
=\exp\!\left(\mu_n+\frac{\sigma_n^2}{2n}+O(n^{-3/2})\right)
\qquad(n\to\infty).
\]
\end{corollary}
\begin{proof}
Theorem~\ref{thm:sandwich} and Lemma~\ref{lem:LnMn-asymp} with any $S>\sup_n\sigma_n$ imply
\[
n!\exp[x_0,\ldots,x_n]
=e^{\mu_n}\left(1+\frac{\sigma_n^2}{2n}+O(n^{-3/2})\right).
\]
Taking logarithms gives the result.
\end{proof}

\section{Useful identities}
\label{sec:identities}

\noindent
In this section, we adopt the shorthand notation $e^{t [x_0, \ldots, x_q]} := f[x_0, \ldots, x_q]$ for $f(x) = e^{t x}$
and $t,\tau,x_0,\dots,x_q\in\mathbb C$.
We begin by briefly recalling several well-known results that will be used below.

The divided difference of any holomorphic function $f(x)$ can be defined over the multiset $[x_0, \ldots, x_q]$ 
using a contour integral~\cite{mccurdy1984, deboor2005},
\begin{equation}
    \label{eq:contour-int-dd}
    f[x_0, \ldots, x_q] \equiv \frac{1}{2\pi i} \oint\limits_{\Gamma} \frac{f(x)}{\prod_{i=0}^q(x - x_i)} \,dx,
\end{equation}
where $\Gamma$ is a positively oriented contour enclosing
all the nodes and lying, together with its interior,
in the domain of $f$.
The divided differences can be shown to satisfy the Leibniz rule~\cite{deboor2005},
\begin{equation}
    \label{eq:leibniz-rule}
    (f \cdot g)[x_0, \ldots, x_q] = \sum_{j=0}^q f[x_0, \ldots, x_j] g[x_j, \ldots, x_q] = \sum_{j=0}^q g[x_0, \ldots, x_j] f[x_j, \ldots, x_q].
\end{equation}
Replacing the variable $x \rightarrow \alpha x$ in~\Cref{eq:contour-int-dd}, we find the rescaling relation,
\begin{equation}
    \label{eq:rescaling-relation}
    \alpha^q e^{t [\alpha x_0, \ldots, \alpha x_q]} = e^{\alpha t [x_0, \ldots, x_q]}.
\end{equation}
Combining~\Cref{eq:rescaling-relation} with Eq. (11) of Ref.~\cite{kunz1965inverse} with $P_m(x) = 1$, we find
\begin{equation}
    \label{eq:laplace-of-dd}
    \lap{\,e^{\alpha t [x_0, \ldots, x_q]}\,}
    \;=\;
    \frac{\alpha^q}{\prod_{j=0}^q \bigl(s-\alpha x_j\bigr)},
\end{equation}
where $\lap{\cdot}$ denotes the Laplace transform in $t \rightarrow s$,
and Eq.~(\ref{eq:laplace-of-dd}) holds in the right half-plane $\Re s \;>\; \max_{0\le j\le q} \Re(\alpha x_j)$,
where the Laplace transform converges.
One can alternatively derive~\Cref{eq:laplace-of-dd} by
substituting \Cref{eq:contour-int-dd} into the Laplace
transform, interchanging the integrals with the contour
chosen in $\Re(\alpha z)<\Re s$ to ensure absolute
convergence, and evaluating the resulting contour
integral by residues.

\begin{theorem}[Convolution]
\label{app-lem:convolution}
Let $j\in\{0,\dots,q-1\}$ and $\beta\ge0$. Then
\begin{equation}
\label{app-eq:convolution-theorem}
\int_0^\beta e^{-\tau [x_{j+1}, \ldots, x_q]}  e^{-(\beta-\tau) [x_0, \ldots, x_j]} \dtau = - e^{-\beta [x_0, \ldots, x_q]}.
\end{equation}
\end{theorem}
\begin{proof}
Define 
\[
f(t):=e^{-t [x_{j+1}, \ldots, x_q]},\qquad 
g(t):=e^{-t [x_0, \ldots, x_j]}.
\]
Their convolution is
\[
(f * g)(t)=\int_0^{t} f(\tau)\, g(t-\tau)\, \dtau,
\]
so Eq.~(\ref{app-eq:convolution-theorem}) states that $(f*g)(\beta)=-e^{-\beta [x_0,\ldots,x_q]}$.
By the convolution property of the Laplace transform and~\Cref{eq:laplace-of-dd}, we find
\begin{multline*}
    \lap{(f * g)(t)} = \lap{f(t)} \lap{g(t)} \\
    = \left( \frac{ (-1)^{q-j-1} }{ \prod_{l=j+1}^q (s + x_l) } \right) \left( \frac{ (-1)^{j} }{ \prod_{m=0}^j (s + x_m) }\right) = \frac{ (-1)^{q-1} }{ \prod_{l=0}^q (s + x_l) } = \lap{ - e^{-t [x_0, \ldots, x_q]} }.
\end{multline*}
By the uniqueness theorem for the Laplace transform and continuity, $(f*g)(t)=-e^{-t [x_0,\ldots,x_q]}$ for all $t\ge 0$.
Evaluating at $t=\beta$ gives Eq.~(\ref{app-eq:convolution-theorem}).
\end{proof} 

\begin{theorem}[Repeated argument sum]
    \label{prop:repatedarg-sum-simp}
    \begin{equation}
        \sum_{j=0}^q e^{-\tau [x_0, \ldots, x_q, x_j]} = -\tau e^{-\tau [x_0, \ldots, x_q]}
    \end{equation}
\end{theorem}
\begin{proof}
    Denote $f(\tau) \equiv e^{-\tau [x_0, \ldots, x_q]}$ and recall  $\lap{e^{-\tau [x_0, \ldots, x_q]}} = (-1)^q / (\prod_{i=0}^q (s + x_i)).$ By the differentiation property of the Laplace transform and Eq.~(\ref{eq:laplace-of-dd}),
        \begin{equation}
            \lap{-\tau f(\tau)} = \pdv{s}{\lap{f(\tau)}} = \sum_{j=0}^q \frac{(-1)^{q+1}}{(s + x_j) \prod_{i=0}^q (s + x_i)} = \lap{\sum_{j=0}^q e^{-\tau [x_0, \ldots, x_q, x_j]}},
        \end{equation}
and the result follows for $\tau\ge0$ by the uniqueness
of the Laplace transform and continuity.
Since both sides are entire functions of $\tau$,
the identity extends to all $\tau\in\mathbb C$.
\end{proof}

\begin{theorem}[Weighted, repeated argument sum]
\label{lemm:weighted-repeatedarg-sum}
    \beq
        \label{eq:weighted-repeatedarg-sum}
        \sum_{j=0}^q x_j e^{-\tau [x_0,\dots,x_q,x_j]} = \left(\tau\frac{\partial}{\partial\tau}-q\right)e^{-\tau [x_0,\dots,x_q]}.
    \eeq
\end{theorem}

\begin{proof}
By Eq.~(\ref{eq:rescaling-relation}),
$e^{-\tau[x_0,\ldots,x_q]}=(-\tau)^q\exp[-\tau x_0,\ldots,-\tau x_q]$. Hence
\begin{multline*}
\Big(\tau\partial_\tau-q\Big)e^{-\tau[x_0,\ldots,x_q]}
=-(-\tau)^{q+1}\,\partial_\tau\exp[-\tau x_0,\ldots,-\tau x_q]\\
=-(-\tau)^{q+1}\frac{1}{2\pi i}\oint_\Gamma\partial_\tau
\frac{e^z}{\prod_{i=0}^{q}(z+\tau x_i)}\,dz
=\sum_{j=0}^{q}x_j\,e^{-\tau[x_0,\ldots,x_q,x_j]}.
\end{multline*}
\end{proof}

\begin{lemma}[Parametric derivative]
    \begin{equation}
    \label{eq:parametric}
        \pdv{\tau} e^{-\tau [x_0, \ldots, x_q]} = - x_0 e^{-\tau [x_0, \ldots, x_q]} - e^{-\tau [x_1, \ldots, x_q]}
\qquad\text{for}\quad q>0.
    \end{equation}
\end{lemma}
\begin{proof}
    First, define $g(x) = -x e^{-\tau x}$ for convenience. By Eq.~\eqref{eq:contour-int-dd}, 
    \begin{align}
        \pdv{\tau} e^{-\tau [x_0, \ldots, x_q]} \equiv \pdv{\tau} \frac{1}{2\pi i} \oint_{\Gamma} \frac{e^{-\tau z}}{\prod_{i=0}^q (z - x_i)} \dz =  \frac{1}{2\pi i} \oint_{\Gamma} \frac{-z e^{-\tau z}}{\prod_{i=0}^q (z - x_i)} \dz \equiv g[x_0, \ldots, x_q].
    \end{align}
    Employing the Leibniz rule, \Cref{eq:leibniz-rule}, we get the desired result. 
\end{proof}

\begin{corollary}
\beq
    \label{eq:computing-weighted-repeatedarg-sum}
    \sum_{j=0}^q x_j e^{-\tau [x_0,\dots,x_q,x_j]} = 
    	\left\{\begin{array}{ll}
            (-x_0\tau - q)e^{-\tau[x_0,\dots,x_q]}-\tau e^{-\tau[x_1,\dots,x_q]}, & \text{\rm for}\quad q > 0,\\
            -\tau x_0 e^{-\tau x_0}, & \text{\rm for}\quad q = 0.
            \end{array}\right.
\eeq
\end{corollary}
\begin{proof}
    This follows from applying Eq.~(\ref{eq:parametric}) to Theorem~\ref{lemm:weighted-repeatedarg-sum}. 
    The case $q=0$ follows directly.
\end{proof}

\begin{theorem}[Weighted, repeated argument, double, less-than sum]
\beq
\sum_{\substack{i,j=0 \\ i \leq j}}^q x_i e^{-\tau [x_0,\dots,x_q,x_i,x_j]} = 
\left(-\frac{\tau^2}{2}\frac{\partial}{\partial\tau}-\sum_{i=0}^q i\cdot \frac{\partial}{\partial x_i}\right) e^{-\tau [x_0,\dots,x_q]}.
\label{eq:c2a}
\eeq
Here $\partial/\partial x_i$ denotes differentiation with respect to the $i$-th argument, 
with all other arguments held fixed.
\end{theorem}

\begin{proof}
We prove this by induction.

Base case.
When $q=0$, we have 
$$
x_0 e^{-\tau [x_0,x_0,x_0]} = x_0 \cdot \frac{1}{2}\frac{\partial^2}{\partial x_0^2} e^{-\tau x_0} 
= \frac{\tau^2 x_0}{2} e^{-\tau x_0} = -\frac{\tau^2}{2}\frac{\partial}{\partial\tau} e^{-\tau x_0},
$$ so Eq.~(\ref{eq:c2a}) holds.

Induction step. 
Assume first that $x_0\ne x_1$.
It follows from Newton's recursion for divided differences that
\begingroup
\allowdisplaybreaks[4]
\begin{align}
\sum_{\substack{i,j=0 \\ i \leq j}}^{q+1} x_i e^{-\tau [x_0,\dots,x_{q+1},x_i,x_j]} &=
\frac{1}{x_0-x_1} \sum_{\substack{i,j=0 \\ i \leq j}}^{q+1} x_i \left( e^{-\tau [x_0,x_2,\dots,x_{q+1},x_i,x_j]} - e^{-\tau[x_1,x_2,\dots,x_{q+1},x_i,x_j]}\right) \notag \\ &= 
\frac{1}{x_0-x_1} \Bigg(
\Big(
\big(-\frac{\tau^2}{2}\frac{\partial}{\partial\tau}-\sum_{i=1}^q i\cdot \frac{\partial}{\partial x_{i+1}}\big)e^{-\tau [x_0,x_2,\dots,x_{q+1}]}  \notag \\
&\qquad + 
x_0 e^{-\tau [x_0,x_2,\dots,x_{q+1},x_0,x_1]}+ 
x_1 \sum_{j=1}^{q+1} e^{-\tau [x_0,\dots,x_{q+1},x_j]}
\Big) \notag \\
-\Big(
\big(&-\frac{\tau^2}{2}\frac{\partial}{\partial\tau}-\sum_{i=1}^q i\cdot \frac{\partial}{\partial x_{i+1}}\big)e^{-\tau [x_1,x_2,\dots,x_{q+1}]} + x_0 \sum_{j=0}^{q+1} e^{-\tau [x_0,x_1,\dots, x_{q+1},x_j]}
\Big)
\Bigg) 
\notag \\ &=
\left(-\frac{\tau^2}{2}\frac{\partial}{\partial\tau}-\sum_{i=0}^q i\cdot \frac{\partial}{\partial x_{i+1}}\right)e^{-\tau [x_0,\dots,x_{q+1}]}
-\sum_{j=1}^{q+1} e^{-\tau [x_0,\dots,x_{q+1},x_j]} \notag \\ &=
\left(-\frac{\tau^2}{2}\frac{\partial}{\partial\tau} -\sum_{i=0}^{q+1} i\cdot \frac{\partial}{\partial x_i}\right) e^{-\tau [x_0,\dots,x_{q+1}]}.
\end{align}
\endgroup

The case $x_0=x_1$ follows by continuity.
Therefore, Eq.~(\ref{eq:c2a}) is true for all $q$.  
\end{proof}

\begin{corollary}
\begin{multline}
\sum_{\substack{i,j=0 \\ i \leq j}}^q x_i e^{-\tau [x_0,\dots,x_q,x_i,x_j]} \\
	= \left\{\begin{array}{ll}
	\left(\tau^2 x_0/2\right) e^{-\tau [x_0,\dots,x_q]} + (\tau^2/2) e^{-\tau [x_1,\dots,x_q]}
	- \sum_{i=1}^q i \cdot e^{-\tau [x_0,\dots,x_q,x_i]}, & \text{\rm for}\quad q > 0,\\
        (\tau^2 x_0/2) e^{-\tau x_0}, & \text{\rm for}\quad q = 0.
        \end{array}\right.
\label{eq:c3a}
\end{multline}
\end{corollary}
\begin{proof}
Eq.~(\ref{eq:c3a}) follows from Eq.~(\ref{eq:parametric}) and Eq.~(\ref{eq:c2a}).
The case $q=0$ follows directly.
\end{proof}

\section{Concluding remarks\label{sec:conc}}

This work identifies and proves several structural properties of exponential divided differences 
that may be useful across analysis, matrix theory, and computation.

Sign-regularity of all orders (including log-submodularity),
supermodularity, and the four-point
inequality provide structural inequalities that may be
useful in the analysis of sensitivity and stability of
divided-difference-based numerical schemes
(e.g., Newton/Leja expansions and $\varphi$-function evaluations in exponential integrators).
The paired-node theorem extends the four-point inequality
to any number of disjoint pairs, permits additional common
nodes, and provides sharp lower bounds in terms of the
pair-averaged configuration and the overall mean.
Its proof is independent of the proof in
Section~\ref{sec:fourpoint}.
The sharp two-sided bounds at fixed mean and variance
characterize the extremal node configurations explicitly.
Together with the uniform asymptotic expansions, they yield
a large-node-count approximation for sequences with
bounded variance.
These results may be useful for deriving error bounds
for interpolation remainders and conditioning estimates.
The closed-form convolution and repeated-argument identities provide
exact relations among divided differences and weighted
sums, potentially useful in Newton/Leja-type expansions,
$\varphi$-function evaluations, and PMR-QMC computations.

Possible directions for future work include applications to
higher-order exponential integrators and quantum Monte Carlo
methods, as well as the following open questions:
(i) To what extent do Theorems~\ref{thm:KnSSR}--\ref{thm:sandwich}
extend to other classes of functions, and under what assumptions?
(ii) Do sharp bounds extend to fixed higher-order moments,
with extremizers given in closed form?

\section*{Acknowledgments}

\subsection*{Funding}
N.E. acknowledges partial support from the Graduate Student Fellowship awarded by the USC Department of Physics and Astronomy in the Dornsife College of Letters, Arts, and Sciences, the Gold Family Fellowship, the ARO MURI grant W911NF-22-S-000, and the U.S. Department of Energy (DOE) Computational Science Graduate Fellowship under Award No. DE-SC0020347 during parts of this work.
I.H. acknowledges support by the Office of Advanced Scientific Computing Research of the U.S. Department of Energy under Contract No. DE-SC0024389. 

\subsection*{Use of generative AI}

During the preparation of Sections~\ref{sec:logsubmod}, \ref{sec:paired-nodes}, and~\ref{sec:bounds}, the authors used OpenAI's ChatGPT and
Anthropic's Claude to check the arguments, to explore alternative proofs, and to improve the exposition.
The authors critically evaluated all AI-assisted suggestions,
verified the mathematical arguments, and reviewed and edited the resulting text.
The authors take full responsibility for the correctness, originality, and
content of this article.

\bibliographystyle{elsarticle-num}
\bibliography{divdiff-paper-v2026}

\end{document}